\documentclass{article}
\usepackage{amsfonts}
\usepackage{amsmath}
\usepackage{maa-monthly}
\usepackage{cancel}

\theoremstyle{theorem}
\newtheorem{theorem}{Theorem}
\newtheorem{lemma}{Lemma}
\newtheorem{corollary}{Corollary}
\newtheorem{proposition}{Proposition}
\theoremstyle{definition}

\newtheorem{remark}{Remark}
\newtheorem{example}{Example}

\newtheorem{conjecture}{Conjecture}

\input{tcilatex}
\begin{document}

\title{Irrationality and Transcendence of Alternating Series Via Continued
Fractions}
\author{Jonathan Sondow}
\maketitle

\begin{abstract}
Euler gave recipes for converting alternating series of two types, I and II,
into \emph{equivalent} continued fractions, i.e., ones whose convergents
equal the partial sums. A condition we prove for irrationality of a
continued fraction then allows easy proofs that $e,\sin1$, and the primorial
constant are irrational. Our main result is that, if a series of type II is
equivalent to a \emph{simple} continued fraction, then the sum is
transcendental and its irrationality measure exceeds~$2$. We construct all $%
\aleph_0^{\aleph_0}=\mathfrak{c}$ such series and recover the transcendence
of the Davison--Shallit and Cahen constants. Along the way, we mention $\pi$%
, the golden ratio, Fermat, Fibonacci, and Liouville numbers, Sylvester's
sequence, Pierce expansions, Mahler's method, Engel series, and theorems of
Lambert, Sierpi\'{n}ski, and Thue-Siegel-Roth. We also make three
conjectures.
\end{abstract}


\section{\emph{Introductio}.}

\label{SEC:introd} 

In a $1979$ lecture on the Life and Work of Leonhard Euler, Andr\'{e} Weil
suggested ``that our students of mathematics would profit much more from a
study of Euler's \textit{Introductio in Analysin Infinitorum}, rather than
of the available modern textbooks'' \cite[p.~xii]{Euler}. The last chapter
of the \textit{Introductio} is ``On Continued Fractions.'' In it, after
giving their form, Euler ``next look[s] for an equivalent expression in the
usual way of expressing fractions'' and derives formulas for the
convergents. He then converts a continued fraction into an \emph{equivalent}
alternating series, i.e., one whose partial sums equal the convergents. He
``can now consider the converse problem. Given an alternating series, find a
continued fraction such that the series representing the value of the
continued fraction is the given series.''

In Proposition~\ref{THM:altproductseries=CF} and Theorem~\ref%
{THM:altproductseries=SCF}, we recall Euler's solutions for alternating
series of two types, I and II. Lemma~\ref{THM:irratCF}, a simplification of
Nathan's theorem on irrationality of a continued fraction, then yields
conditions for irrationality of the sum of a type I or II series. They
easily imply the irrationality of $e,\sin1,$ and the shifted-Fermat-number
and primorial constants, and give a simple proof of Sierpi\'{n}ski's theorem.

Our main result is that, if a type~II series is equivalent to a \emph{simple}
continued fraction, then the sum has irrationality measure greater than~$2,$
and so must be transcendental, by the Thue-Siegel-Roth theorem on rational
approximations to algebraic numbers.

Corollary~\ref{COR:(iii)} constructs all such series and shows that their
sums form a continuum of distinct transcendental numbers, including the
Davison-Shallit constant.

Corollary~\ref{COR:SCFforCgen} gives explicitly the simple continued
fractions for ``naturally-occurring'' transcendental numbers in a
doubly-infinite family which contains Cahen's constant.

Finally, Proposition~\ref{PROP:K_k} provides irrationality and transcendence
conditions for families of \emph{non}-alternating series, including the
Kellogg-Curtiss constant. Here the proofs involve partial sums instead of
continued fractions.

Along the way, we encounter $\pi$, Fibonacci and golden rectangle numbers,
an alternating Liouville constant, Sylvester's sequence, Pierce expansions,
Mahler's method, and Engel series. We also make three conjectures; one on~$%
e^{-1}$ is an analog of Sondow's conjecture on~$e,$ recently proven by
Berndt, Kim, and Zaharescu.

The rest of the paper is organized as follows. Lemma~\ref{THM:irratCF} and
Proposition~\ref{THM:altproductseries=CF} are in Section~\ref{SEC:lemma};
Theorem~\ref{THM:altproductseries=SCF}, Corollary~\ref{COR:(iii)}, and
Conjectures~1 and~2 are in Section~\ref{SEC:main}; Corollary~\ref%
{COR:SCFforCgen} is in Section~\ref{SEC:Sylvester}; and Proposition~\ref%
{PROP:K_k} and Conjecture~3 are in Section~\ref{Kellogg and Curtiss}.


\section{continued fractions and irrationality.}

\label{SEC:lemma} 

In 1761 Lambert \cite{Lambert} derived a continued fraction for $\tan x$ and
showed that its value is irrational for rational $x\neq0$. Since $\tan \frac{%
\pi}{4}=1$ is rational, Lambert had established that $\pi$ \emph{is
irrational}. For modern treatments of his proof, see \cite[\S 3.6]%
{DuverneyBook} and \cite{Laczkovich}.

Let us denote the positive integers by $\mathbb{N}$ and the rational numbers
by $\mathbb{Q}$. Lemma~\ref{THM:irratCF} provides a sufficient condition for
irrationality of the value of a continued fraction with all elements in~$%
\mathbb{N}$. (Lambert's has both positive and negative elements.) The
statement and quick proof are simplifications of Nathan's theorem in \cite%
{Nathan}.

\begin{lemma}[Irrationality Lemma]
\label{THM:irratCF} Let $\alpha$ be the value of a continued fraction 
\begin{equation*}
\alpha=\frac{b_1}{a_1+\displaystyle{\frac{b_2}{a_2 +\cdot_{\displaystyle{\,
\cdot}_{\displaystyle{\, \cdot}}}}}}, 
\end{equation*}
where $a_n\in\mathbb{N}$ and $b_n\in\mathbb{N}$ and $a_n\ge b_n$ for $%
n=1,2,3,\dotsc$. Then $\alpha\not\in\mathbb{Q}$.
\end{lemma}

\begin{proof}
If $\alpha\in\mathbb{Q},$ define the $n$th ``tail'' of $\alpha$ to be the value of the continued fraction
\begin{equation} \label{EQ:irrat}
\alpha_n := \frac{b_{n+1}}{a_{n+1}+\displaystyle{\frac{b_{n+2}}{a_{n+2} +\cdot_{\displaystyle{\, \cdot}_{\displaystyle{\, \cdot}}}}}}, \quad \text{so} \quad  \alpha_n = \frac{b_{n+1}}{a_{n+1}+\alpha_{n+1}},
\end{equation}
for all $n\ge0$. The hypotheses ensure that $0<\alpha_n<1$ for all $n\ge0$.  
As $\alpha_0=\alpha,$ and $\alpha_n\in\mathbb{Q}$ implies $\alpha_{n+1}\in\mathbb{Q},$ we can write $\alpha_n=u_n/v_n,$ where $u_n$ and $v_n$ are coprime positive integers with $u_n<v_n$. Thus from \eqref{EQ:irrat} we get
$$\frac{u_{n+1}}{v_{n+1}} =\alpha_{n+1}= \frac{v_nb_{n+1}-u_na_{n+1}}{u_n},$$
so $u_{n+1}<v_{n+1}\le u_{n}$. But then $(u_n)_{n\ge0}$ is a strictly decreasing, infinite sequence of positive integers, which is impossible. Therefore, $\alpha\not\in\mathbb{Q}$. 
\end{proof}

For instance, if $a_n=1$ and $b_n=1$ for all $n,$ then by Lemma \ref%
{THM:irratCF} 
\begin{equation*}
\alpha= \frac{1}{1+\displaystyle{\frac{1}{1 +\cdot_{\displaystyle{\, \cdot}_{%
\displaystyle{\, \cdot}}}}}} = \frac{1}{1+\alpha} >0 \quad \implies \quad
\alpha=\frac{\sqrt{5}-1}{2} \not\in\mathbb{Q}. 
\end{equation*}
Thus \emph{the golden ratio $\varphi:=\alpha^{-1}$ is irrational}. For more
on $\varphi,$ see Examples \ref{EX:Pierce} and~\ref{EX:phi2}.

Lemma \ref{THM:irratCF} generalizes the irrationality of an infinite \emph{%
simple} continued fraction, i.e., one with all \emph{partial numerators} $%
b_n=1$ and all \emph{partial quotients} (or \emph{partial denominators}) $%
a_n\in\mathbb{N}$.

Our hypothesis $a_n\ge b_n$ is weaker than Nathan's $a_n> b_n$. Ours is also 
\emph{sharp}: with the even weaker hypothesis $a_n\ge b_n-1,$ the lemma
would be false, e.g., 
\begin{equation*}
\alpha= \frac{2}{1+\displaystyle{\frac{2}{1+\cdot_{\displaystyle{\, \cdot}_{%
\displaystyle{\, \cdot}}}}}} = \frac{2}{1+\alpha} > 0 \quad \implies \quad
\alpha=1\in\mathbb{Q}. 
\end{equation*}

Lemma \ref{THM:irratCF} holds more generally when $a_n\ge b_n$ for all \emph{%
sufficiently large}~$n$. There is also a condition for irrationality of a
continued fraction with both positive and negative integers $a_n$ and $b_n,$
namely, that $\left\vert a_n\right\vert\ge \left\vert b_n\right\vert+1;$
see, e.g., \cite[\S 3.6]{DuverneyBook}. We have chosen simplicity over
generality here and elsewhere in the paper.

We now apply the Irrationality Lemma to our first kind of alternating
series, type I.

\begin{proposition}
\label{THM:altproductseries=CF} Let $B_0<B_1< B_2< \dotsb$ be positive
integers.

\noindent\textrm{(i).} Then there is an equivalence 
\begin{equation}  \label{EQ:altseries=CF}
\alpha:=\frac{1}{B_0}-\frac{1}{B_1}+\frac{1}{B_2}-\dotsb \cong\frac{1}{B_0+%
\displaystyle{\frac{B_0^2}{B_1-\displaystyle{B_0+\frac{B_1^2}{B_2-B_1+\cdot_{%
\displaystyle{\, \cdot}_{\displaystyle{\, \cdot}}}}}}}}.
\end{equation}

\noindent\textrm{(ii).} Suppose that 
\begin{equation}  \label{EQ:Bineq}
\ B_{n+1}\ge B_n(B_n+1)\ \text{ for all }\ n\ge0.
\end{equation}
Then the sum $\alpha$ is irrational.
\end{proposition}

\begin{proof}
\noindent (i).
Euler establishes the equivalence in \cite[\S 369]{Euler};
 for example,
 $$\frac{1}{B_0}-\frac{1}{B_1}
= \frac{1}{B_0+\displaystyle{\frac{B_0^2}{B_1-\displaystyle{B_0}}}}.
$$

\noindent (ii).
Set $a_1=B_0,b_1=1,$ $a_{n+1}=B_n-B_{n-1},$ and $b_{n+1}=B_{n-1}^2$ for $n\ge1$. Then \eqref{EQ:Bineq} guarantees that $a_n\ge b_n$ for all $n,$ so by Lemma~\ref{THM:irratCF} the value of the continued fraction in \eqref{EQ:altseries=CF} is irrational. By (i), that value equals the sum $\alpha,$ so $\alpha\not\in\mathbb{Q}$.
\end{proof}

Proposition \ref{THM:altproductseries=CF} provides an easy proof of \emph{%
Sierpi\'{n}ski's theorem}, which states that, \emph{if \eqref{EQ:Bineq}
holds with all $B_n\in\mathbb{N},$ then} $\alpha:=\sum_{n=0}^{%
\infty}(-1)^nB_n^{-1} \not\in\mathbb{Q}$. Sierpi\'{n}ski \cite{Sierpinski}
(see also Cahen \cite{Cahen}) showed moreover that \emph{such a
representation of any irrational number $\alpha$ in $(0,1)$ exists and is
unique}. For extensions of his theorem, see Badea \cite{Badea2}, Duverney 
\cite{Duverney}, and Nyblom \cite{Nyblom}.

Note that \emph{part \textrm{(ii)} and Sierpi\'{n}ski's theorem are sharp}:
if $B_{n+1}+1= B_n( B_n+1) $ for all $n\ge0,$ then $(B_{n+1}+1)^{-1}=
B_n^{-1} -(B_n+1)^{-1},$ so by telescoping 
\begin{equation*}
\sum_{n=0}^{\infty}\frac{(-1)^n}{B_n} = \sum_{n=0}^{\infty} \left( \frac{%
(-1)^n}{B_n+1} + \frac{(-1)^n}{B_{n+1}+1}\right) =\frac{1}{B_0+1}\in\mathbb{Q%
}. 
\end{equation*}

\begin{example}
\label{EX: Fermat} The \emph{Fermat numbers} $F_n=2^{2^n}+1$ form the
sequence \cite[A000215]{Sloane} 
\begin{equation*}
(F_n)_{n\ge0}= 3, 5, 17, 257, 65537, 4294967297, 18446744073709551617,\dotso%
. 
\end{equation*}
Let us define the \emph{shifted-Fermat-number constant} $F$ to be the
alternating sum of reciprocals of the numbers $F_n-2$ (for them, see \cite[%
A051179]{Sloane}) 
\begin{align*}
F:=\ \sum_{n=0}^{\infty}\frac{(-1)^{n}}{F_n-2} = \sum_{n=0}^{\infty}\frac{%
(-1)^{n}}{ 2^{2^n}-1} = 1-\frac{1}{3}+\frac{1}{15}-\frac{1}{255}+\dotsb={%
0.7294270}\dotsc.
\end{align*}
The numbers $B_n:=2^{2^n}-1$ satisfy \eqref{EQ:Bineq}, so \emph{the
shifted-Fermat-number constant $F$ is irrational}. For a generalization with
a different proof, take $\epsilon=-1$ in \cite[Corollary~3.3]{Duverney}. We
return to $F$ in Example~\ref{EX: Fermat2}.
\end{example}

The next section studies irrationality and transcendence of our second kind
of alternating series, type~II, which is a special case of type~I.


\section{simple continued fractions and Transcendence.}

\label{SEC:main} 

Our main results are Theorem~\ref{THM:altproductseries=SCF} and Corollaries~%
\ref{COR:(iii)} and~\ref{COR:SCFforCgen}. We denote the algebraic numbers by 
$\mathbb{A}$ (others denote them by $\overline{\mathbb{Q}},$ the algebraic
closure of $\mathbb{Q}$).

\begin{theorem}
\label{THM:altproductseries=SCF} Fix positive integers $A_0,A_1,A_2,\dotsc,$
with $A_n\ge2$ for all $n\ge1$.

\noindent\textrm{(i).} For any positive real numbers $x_0,x_1,x_2,\dotsc,$
we have the equivalence between an alternating series and a continued
fraction 
\begin{equation}  \label{EQ:equiv}
\alpha: = \sum_{n=0}^{\infty}\frac{(-1)^{n}}{A_0A_1\dotsb A_n} \cong \frac{%
x_0}{A_0x_0+\displaystyle{\frac{A_0x_0x_1}{(A_1-\displaystyle{1)x_1+\frac{%
A_1x_1x_2}{(A_2-\displaystyle{1)x_2 +\cdot_{\displaystyle{\, \cdot}_{%
\displaystyle{\, \cdot}}}}}}}}}.
\end{equation}

\noindent\textrm{(ii).} If $A_{n+1}>A_n$ for all $n\ge0,$ then $\alpha$ is
irrational.

\noindent\textrm{(iii).} If the continued fraction is simple for some $%
x_0,x_1,x_2,\dotsc,$ then $\alpha$ is a transcendental number, with
irrationality measure \mbox{$\mu(\alpha)\ge2.5$}.
\end{theorem}

The \emph{irrationality measure} (or \emph{irrationality exponent}) $%
\mu(\rho)$ of a real number~$\rho$ is defined as (see \cite{BBC, BRS, BC}, 
\cite[\S 1.4]{BPSZ}, \cite[Chapter~9]{DuverneyBook}, \cite[\S 2.22]{Finch}, 
\cite{Sondow}) 
\begin{equation}  \label{EQ:mu def}
\mu(\rho) := \sup\left\{\mu>0: 0 < \left\vert \rho - \frac{p}{q} \right\vert
< \frac{1}{q^{\mu}} \text{ for infinitely many } \frac{p}{q} \in\mathbb{Q}
\right\}.
\end{equation}
By the famous \emph{Thue-Siegel-Roth theorem} \cite{Adam}, \cite[p.~22]{BPSZ}%
, \cite[p.~147]{DuverneyBook}, \cite[p.~172]{Finch}, \cite[p.~176]{HW} 
\begin{equation*}
\mu(\rho) 
\begin{cases}
=1 \quad \text{if $\rho$ is rational}, \\ 
=2 \quad \text{if $\rho$ is irrational, but algebraic}, \\ 
\ge2 \quad \text{if $\rho$ is transcendental}.%
\end{cases}
\end{equation*}

\begin{proof}[Proof of Theorem~\ref{THM:altproductseries=SCF}]
\noindent(i). Apply Proposition~\ref{THM:altproductseries=CF}, part~(i), with $B_n:=A_0A_1\dotsb A_n$ for $n\ge0$. Since $B_{n}-B_{n-1}=(A_{n}-1)A_0A_1\dotsb A_{n-1},$ cancelling the common factors $A_0,A_0A_1,$ $A_0A_1A_2,\dotsc$ in the resulting continued fraction gives
\begin{align*} \label{EQ:equiv1}
&\frac{1}{A_0}-\frac{1}{A_0A_1}+\frac{1}{A_0A_1A_2}-\dotsb 
\cong\frac{1}{A_0+\displaystyle{\frac{A_0^{\cancel{2}}}{(A_1-\displaystyle{1)\cancel{A_0}+\frac{A_0^{\cancel{2}}A_1^2}{(A_2-\displaystyle{1)A_0A_1 +\cdot_{\displaystyle{\, \cdot}_{\displaystyle{\, \cdot}}}}}}}}}\nonumber\\
&\cong\frac{1}{A_0+\displaystyle{\frac{A_0}{A_1-\displaystyle{1+\frac{\cancel{A_0}A_1^{\cancel{2}}}{(A_2-\displaystyle{1)\cancel{A_0A_1} +\cdot_{\displaystyle{\, \cdot}_{\displaystyle{\, \cdot}}}}}}}}}\nonumber 
\cong \frac{1}{A_0+\displaystyle{\frac{A_0}{A_1-\displaystyle{1+\frac{A_1}{A_2-\displaystyle{1 +\cdot_{\displaystyle{\, \cdot}_{\displaystyle{\, \cdot}}}}}}}}},
\end{align*}
where ``$\cong$'' between two continued fractions means they are {\em equivalent}, i.e., they
have the same convergents (see \cite[p.~25]{DuverneyBook}; for two numerical continued fractions which are equivalent but not equal, see Example~\ref{EX:primordial} below).
This proves the special case of (i) in which all $x_n=1$ (compare to \cite[\S370]{Euler}).
The general case follows by cancelling the common factors $x_0,x_1,x_2,\dotsc$ in \eqref{EQ:equiv}.

\noindent(ii).
In Lemma \ref{THM:irratCF}, we take $a_1:=A_0,b_1:=1,$ $a_n:=A_{n-1}-1,$ and \mbox{$b_n:=A_{n-2}$} for $n\ge2$.
Then $A_{n+1}>A_n$ implies $a_n\ge b_n$ for all $n\ge1,$ so $\alpha\not\in\mathbb{Q}$.

\noindent(iii). (Compare to the proof of \cite[Theorem~3]{SD}.)
Redefining $a_1,a_2,\dotsc$, we write the simple continued fraction for $\alpha,$ and its $n$th convergent, as usual as
$$\alpha = 0+\frac{1}{a_1+\displaystyle{\frac{1}{a_2+\displaystyle{\cdot_{\displaystyle{\, \cdot}_{\displaystyle{\, \cdot}}}}}}} = [0,\,a_1,\,a_2,\,\dotsc]
\quad \text{and} \quad \frac{p_n}{q_n} =  [0,\,a_1,\,a_2,\,\dotsc,\,a_n].
$$
The hypothesis in (iii)
means that
\begin{equation} \label{EQ:SCF}
\sum_{i=0}^{n}\frac{(-1)^{i}}{A_0A_1\dotsb A_i}  = \frac{p_{n+1}}{q_{n+1}} \ \text{ for }  \ n\ge0.
\end{equation}
A classical theorem \cite[Theorem 150]{HW} and relation \eqref{EQ:SCF} imply, respectively, that
\begin{equation} \label{EQ:classic}
 \frac{(-1)^{n}}{q_{n}q_{n+1}}= \frac{p_{n+1}}{q_{n+1}} - \frac{p_{n}}{q_{n}} = \frac{(-1)^{n}}{A_0A_1\dotsb A_n}
\end{equation}
for $n\ge1$. Hence $q_{n}q_{n+1} = A_0A_1\dotsb A_n;$ since $q_0=1$ and $q_1=A_0,$ this also holds for $n=0$.
It follows that the divisibility $q_nq_{n+1}\mid q_{n+1}q_{n+2}$ holds; hence $q_n\mid q_{n+2}$. A~standard identity \cite[Theorem~149]{HW} is
\begin{equation} \label{EQ:qn}
q_{n+2} = a_{n+2}q_{n+1}+q_n,
\end{equation}
so $q_n\mid a_{n+2}q_{n+1}$. Multiplying \eqref{EQ:classic} by $q_{n}q_{n+1},$ we deduce that $\gcd(q_n,q_{n+1})=1,$ so $q_n\mid a_{n+2}$. Define $w_0,w_1,\dotsc$ in $\mathbb{N}$ by $w_0=a_1$ and $ w_{n+1}q_n = a_{n+2}$
for $n\ge0$. By a ``simple lemma'' \cite[Lemma~2]{SD}, 
\begin{equation} \label{EQ:sqrt}
w_nq_{n-1}\ge\sqrt{q_n} \quad \text{for infinitely many } n.
\end{equation}
Now, from \eqref{EQ:SCF}, a classical inequality \cite[p.~24]{BPSZ}, the equality $ a_{n+1} = w_{n}q_{n-1},$ and \eqref{EQ:sqrt},
respectively, we see that
$$0 < \left\vert \alpha - \frac{p_{n}}{q_{n}}  \right\vert < \frac{1}{a_{n+1}q^2_{n}} = \frac{1}{w_{n}q_{n-1}q^2_{n}} \le \frac{1}{q^{5/2}_{n}}$$
infinitely often.
This and definition \eqref{EQ:mu def} imply $\mu(\alpha)\ge2.5$. By the Thue-Siegel-Roth theorem, $\mu(\rho)\le2$ if $\rho\in\mathbb{A},$ 
so $\alpha\not\in\mathbb{A}$. This completes the proof of Theorem~\ref{THM:altproductseries=SCF}.
\end{proof}

Note that \emph{the hypothesis in \textrm{(ii)} is sharp}: if $A_n=A_0>1$
for all $n>0,$ then the series in \eqref{EQ:equiv} is geometric, with sum $%
\alpha=(A_0+1)^{-1}\in\mathbb{Q}$. Also, in (ii) the inequality $A_{n+1}>A_n$
is much weaker than that in \eqref{EQ:Bineq} with $B_n=A_0A_1\dotsb A_n,$
which amounts to $A_{n+1}>A_0A_1\dotsb A_n$. Compare Examples~\ref{EX:
Fermat} and ~\ref{EX: Fermat2}.

For any strictly increasing sequence of positive integers $%
A_0<A_1<A_2<\dotsb,$ finite or infinite, the alternating sum 
\begin{equation*}
\alpha := \frac{1}{A_0} -\frac{1}{A_0A_1}+ \frac{1}{A_0A_1A_2}-\dotsb 
\end{equation*}
is called the \emph{Pierce expansion} of~$\alpha$. \emph{Any number $%
\alpha\in(0,1)$ has a unique Pierce expansion, which is infinite if, and
only if, $\alpha$ is irrational} \cite{PVB98,PVB,Pierce,Sierpinski}. The
``only if'' part follows immediately from (ii).

\begin{example}
\label{EX:Pierce} The Pierce expansion of $\varphi^{-1}$ begins \cite[%
A118242, A006276]{Sloane} 
\begin{equation*}
\frac{1}{\varphi} = \frac11 - \frac{1}{1\cdot2} + \frac{1}{1\cdot2\cdot4} - 
\frac{1}{1\cdot2\cdot4\cdot17} + \frac{1}{1\cdot2\cdot4\cdot17\cdot19}
-\dotsb. 
\end{equation*}
As $\varphi^{-1}\in\mathbb{A},$ we see that \emph{the hypothesis in \textrm{%
(iii)} cannot be omitted}. Combined with the next example, this shows that, 
\emph{if the Pierce expansion of $\alpha\not\in\mathbb{Q}$ is not equivalent
to a simple continued fraction, then $\alpha\in\mathbb{A}$ is possible, but
so is $\alpha\not\in\mathbb{A}$}.
\end{example}

\begin{example}
\label{EX:e} Euler \cite[p.~325]{Euler} says, ``Something especially
deserving of our attention is the number~$e\dotso$.'' The Taylor series $e^t
= \sum_{n=0}^{\infty} t^nn!^{-1}$ and (i) lead to the Pierce expansion of $%
e^{-1}$ and the equivalence 
\begin{equation}  \label{EQ:e}
e^{-1}=\sum_{n=2}^{\infty} \frac{(-1)^n}{2\cdot3\cdot4\dotsb n} \cong \frac{%
x_0}{2x_0+\displaystyle{\frac{2x_0x_1}{2x_1+\displaystyle{\frac{3x_1x_2}{%
3x_2+\displaystyle{\frac{4x_2x_3}{4x_3+\cdot_{\displaystyle{\, \cdot}_{%
\displaystyle{\, \cdot}}}}}}}}}}.
\end{equation}
Part (ii) now gives an easy proof that \emph{$e$ is irrational}. The Taylor
series for $\sin t$ and $\cos t$ lead to similar proofs that $\sin\frac1k$ 
\emph{and $\cos\frac1k$ are irrational for all} $k\in\mathbb{N}$.

From \eqref{EQ:e} we also see that a strong converse to (iii) is not true.
Namely, \emph{although $e^{-1}\not\in\mathbb{A}$} (because $e\not\in\mathbb{A%
}$ by Hermite \cite[\S 12.14]{DuverneyBook}), \emph{the type II series for $%
e^{-1}$ in \eqref{EQ:e} is not equivalent to a simple continued fraction}.
Indeed, when $x_0,x_1,\dotso$ are chosen so that all partial numerators in
the continued fraction for $e^{-1}$ in \eqref{EQ:e} equal~$1$ 
\begin{equation}  \label{EQ:not simple}
\frac1e = \frac{1}{2+\displaystyle{\frac{2(1/2)}{2(1/2)+\displaystyle{\frac{%
3(1/2)(2/3)}{3(2/3)+\displaystyle{\frac{4(2/3)(3/8)}{4(3/8)+\cdot_{%
\displaystyle{\, \cdot}_{\displaystyle{\, \cdot}}}}}}}}}} = \frac{1}{2+%
\displaystyle{\frac{1}{1+\displaystyle{\frac{1}{2+\displaystyle{\frac{1}{%
\displaystyle{(3/2)}+\cdot_{\displaystyle{\, \cdot}_{\displaystyle{\, \cdot}%
}}}}}}}}}
\end{equation}
the partial quotients do not all lie in $\mathbb{N}$. For a weaker converse
to (iii), which is also not true, see Example \ref{EX:lambda}.
\end{example}

By \eqref{EQ:not simple}, the simple continued fraction for $e^{-1}$ begins $%
e^{-1} = [0,2,1,2,\dotsc]$. From \eqref{EQ:e} (or by inspection), the first
four convergents are also partial sums of the Taylor series $e^{-1}=
\sum_{n=0}^{\infty} (-1)^nn!^{-1}$.

\begin{conjecture}
\label{CONJ:wpsa} Only four partial sums of the Taylor series for $e^{-1}$
are convergents to $e^{-1},$ namely, $0, 1/2, 1/3,$ and $3/8$.
\end{conjecture}

Conjecture \ref{CONJ:wpsa} is an analog for $e^{-1}$ of the fact that \emph{%
only two partial sums of the Taylor series for $e$ are convergents to $e,$
namely, $2$ and} $8/3$. This property of~$e$ was conjectured by Sondow \cite%
{Sondow}, partially proven by him and Schalm \cite{SondowSchalm}, and
recently proven in full by Berndt, Kim, and Zaharescu \cite{BKZ}.

\begin{example}
\label{EX:primordial} An analog of series \eqref{EQ:e} for $e^{-1},$ with
the factorial $n!$ replaced by the primorial $p_n\#,$ is ``the constant
obtained through Pierce retro-expansion of the prime sequence'' \cite[A132120%
]{Sloane}, which we dub the \emph{primorial constant} 
\begin{align*}
P:=\sum_{n=1}^{\infty}\frac{(-1)^{n-1}}{p_n\#} &= \frac{1}{2}-\frac{1}{%
2\cdot3}+\frac{1}{2\cdot3\cdot5}-\frac{1}{2\cdot3\cdot5\cdot7}+\frac{1}{%
2\cdot3\cdot5\cdot7\cdot11}-\dotsb \\
&=\frac{1}{2}-\frac{1}{6}+\frac{1}{30}-\frac{1}{210}+\frac{1}{2310}-\dotsb =
0.3623062223\dotso.  \notag
\end{align*}
Proposition~\ref{THM:altproductseries=CF}, part (i), and Theorem~\ref%
{THM:altproductseries=SCF}, parts (i) and (ii), imply that 
\begin{align*}
P = \frac{1}{2+\displaystyle{\frac{2^2}{4+\displaystyle{\frac{6^2}{24+%
\displaystyle{\frac{30^2}{180+\displaystyle{\frac{210^2}{2100+\cdot_{%
\displaystyle{\, \cdot}_{\displaystyle{\, \cdot}}}}}}}}}}}} \cong \frac{1}{2+%
\displaystyle{\frac{2}{2+\displaystyle{\frac{3}{4+\displaystyle{\frac{5}{6+%
\displaystyle{\frac{7}{10+\cdot_{\displaystyle{\, \cdot}_{\displaystyle{\,
\cdot}}}}}}}}}}}} \not\in\mathbb{Q}.
\end{align*}
\end{example}

\begin{conjecture}
The primorial constant $P$ is transcendental.
\end{conjecture}

\begin{example}
\label{EX: Fermat2} By induction, for $n\ge0$ the shifted Fermat number $%
F_n-2$ can be factored as the product of all smaller Fermat numbers 
\begin{equation}  \label{EQ:Polya}
F_{n}-2=2^{2^n}-1=\prod_{k=0}^{n-1}\left(2^{2^k}+1\right)= F_0F_1\dotsb
F_{n-1},
\end{equation}
where the empty product equals~$1$ when $n=0$. (From \eqref{EQ:Polya} P\'{o}%
lya deduced that $F_0,F_1,F_2,\dotso$ are pairwise coprime, thereby giving
an alternate proof to Euclid's theorem on the infinitude of the primes \cite[%
\S 2.4]{HW}.) The constant $F$ in Example~\ref{EX: Fermat} thus has Pierce
expansion 
\begin{align*}
F= \sum_{n=0}^{\infty}\frac{(-1)^{n}}{F_0 F_1\dotsb F_{n-1}} =\frac{1}{1}-%
\frac{1}{1\cdot3}+\frac{1}{1\cdot3\cdot5}-\frac{1}{1\cdot3\cdot5\cdot17}%
+\dotsb.
\end{align*}
Part (ii) of Theorem~\ref{THM:altproductseries=SCF} now gives a second proof
that $F\not\in\mathbb{Q}$. Moreover, parts (i) of Proposition~\ref%
{THM:altproductseries=SCF} and Theorem~\ref{THM:altproductseries=SCF} yield
the equivalent continued fractions 
\begin{align*}
F= \frac{1}{1+\displaystyle{\frac{1^2}{2+\displaystyle{\frac{3^2}{12+%
\displaystyle{\frac{15^2}{240+\cdot_{\displaystyle{\, \cdot}_{\displaystyle{%
\, \cdot}}}}}}}}}} \cong \frac{1}{1+\displaystyle{\frac{1}{2+\displaystyle{%
\frac{3}{4+\displaystyle{\frac{5}{16+\cdot_{\displaystyle{\, \cdot}_{%
\displaystyle{\, \cdot}}}}}}}}}}.
\end{align*}
Theorem~\ref{THM:altproductseries=SCF} does not yield $F\not\in\mathbb{A},$
but Duverney \cite{DuverneyUnpub} has proven it by other methods.
\end{example}

\begin{remark}
\emph{Non-}alternating series involving $F_n$ have also been studied. In $%
1963,$ Golomb \cite{Golomb2} proved that \emph{the sum $G:=\sum_{n=0}^{%
\infty}F_n^{-1}$ is irrational}. Two years later, Mahler~\cite{Mahler}
remarked that \emph{$G$ is in fact transcendental}, as a consequence of a
general theorem he proved in $1929$---see \cite[pp.~194--195]{DuverneyTransc}%
. (\emph{Mahler's method} \cite[\S 12.3]{DuverneyBook} proves the
transcendence of values, at certain algebraic points, of functions that
satisfy a type of functional equation.) Recently, Coons \cite{Coons} showed
that $G$ \emph{has irrationality measure} $\mu(G)=2$. In the pre-Mahler year 
$1916$, Kempner \cite{Kempner} proved that \emph{the number $\kappa:=
\sum_{n=0}^{\infty}(F_n-1)^{-1}=\sum_{n=0}^{\infty}2^{-2^n}$ is
transcendental}; see Adamczewski \cite{Adam} for five proofs with
interesting comments. (The second proof applies Mahler's method to the
function $f(x):=\sum_{n=0}^{\infty}x^{2^n},$ which is defined when $%
\left\vert x\right\vert < 1,$ satisfies the functional equation $%
f(x^2)=f(x)-x,$ and has the value $f(1/2)=\kappa$.)
\end{remark}

The next example shows that \emph{the sufficient condition for transcendence
of the sum of a type~II series in Theorem~\ref{THM:altproductseries=SCF}
does not extend to the more general type~I series in Proposition}~\ref%
{THM:altproductseries=SCF}.

\begin{example}
\label{EX:phi2} Let $(f_n)_{n\ge0}=1,1,2,3,5,8,13,\dotsc$ be the positive 
\emph{Fibonacci numbers} \cite[A000045]{Sloane}, defined by $f_0=1, f_1=1,$
and $f_{n+1}=f_n+f_{n-1}$ for $n\ge1$. The product $B_n:=f_nf_{n+1}$ is a 
\emph{golden rectangle number} \cite[A001654]{Sloane}. The difference
between successive golden rectangle numbers is a square: 
\begin{equation}  \label{EQ:square}
B_{n} -B_{n-1} = f_{n}f_{n+1}-f_{n-1}f_{n} = f_{n}(f_{n+1}-f_{n-1}) =
f_{n}^2.
\end{equation}
Therefore, using Proposition~\ref{THM:altproductseries=SCF}, part (i), and
cancelling common factors $f_1^2,f_2^2,\dotsc,$ we obtain the equivalence 
\begin{equation*}
\alpha:=\sum_{n=0}^{\infty}\frac{(-1)^{n}}{f_nf_{n+1}} \cong \frac{1}{f_0f_1+%
\displaystyle{\frac{f_0^2\cancel{f_1^2}}{\cancel{f_1^2}\displaystyle{+\frac{%
\cancel{f_1^2}\cancel{f_2^2}}{\cancel{f_2^2}+\cdot_{\displaystyle{\, \cdot}_{%
\displaystyle{\, \cdot}}}}}}}} = [0,1,1,1,\dotsc]. 
\end{equation*}
The latter is the simple continued fraction expansion of $\alpha
=\varphi^{-1}\in\mathbb{A}$. This shows that, \emph{given $%
B_0<B_1<B_2<\dotsb $ in $\mathbb{N},$ the sum of the series $%
\alpha:=\sum_{n=0}^{\infty}(-1)^nB_n^{-1}$ might not be transcendental, even
if the series is equivalent to a simple continued fraction}. (However, \emph{%
if in addition $B_{n-1}$ divides $B_{n}$ for all $n\ge1,$ then} $%
\alpha\not\in\mathbb{A},$ by Theorem~\ref{THM:altproductseries=SCF} with $%
A_0:=B_0$ and $A_n:=B_{n}/B_{n-1}$ for $n\ge1$.)
\end{example}

\begin{remark}
Example \ref{EX:phi2} is a special case of the following well-known fact. 
\emph{For any irrational number $\rho$ with simple continued fraction
expansion $\rho=[a_0,a_1,a_2,\dotsc]$ and $n$th convergent $p_n/q_n,$ there
is an equivalence} 
\begin{equation*}
\rho = a_0+\sum_{n=0}^{\infty}\frac{(-1)^{n}}{q_nq_{n+1}} \cong
[a_0,a_1,a_2,\dotsc] . 
\end{equation*}
(\emph{Proof.} Replacing $\rho$ with $\rho-a_0,$ we may assume that $a_0=0$.
Note that $q_0=1$. Setting $B_n=q_{n}q_{n+1},$ we use \eqref{EQ:qn} to get $%
B_{n} -B_{n-1} = a_{n+1}q_{n}^2,$ generalizing relation \eqref{EQ:square}.
The rest of the proof is like the argument in Example~\ref{EX:phi2}, and is
omitted.)
\end{remark}

By part (i) of Theorem~\ref{THM:altproductseries=SCF}, \emph{if the
continued fraction in \eqref{EQ:equiv} is simple for some $x_0,x_1,\dotsc,$
then the series in \eqref{EQ:equiv} is equivalent to a simple continued
fraction}, i.e., \eqref{EQ:SCF} holds. Conversely, it is not hard to show by
induction that, \emph{if \eqref{EQ:SCF} holds, then the continued fraction
in \eqref{EQ:equiv} is simple for some} $x_0,x_1,\dotsc$. For instance, if
the partial sums $A_0^{-1}$ and \mbox{$A_0^{-1}- (A_0A_1)^{-1}$} equal the
convergents $a_1^{-1}$ and $(a_1+ a_2^{-1})^{-1},$ respectively, then $%
A_0=a_1$ and \mbox{$(A_1-1)A_0^{-1}=a_2\in\mathbb{N}$}, so the choices $%
x_0=1 $ and $x_1=A_0^{-1}$ give the finite simple continued fraction 
\begin{equation*}
\frac{x_0}{A_0x_0+\displaystyle{\frac{A_0x_0x_1}{(A_1-\displaystyle{1)x_1}}}}
= \frac{1}{A_0+\displaystyle{\frac{1}{(A_1-\displaystyle{1)A_0^{-1}}}}} =
[0,a_1,a_2]. 
\end{equation*}

We now give a method for constructing all examples of Theorem~\ref%
{THM:altproductseries=SCF}, part (iii).\newline

\begin{corollary}
\label{COR:(iii)} \textrm{(i).} Construct a sequence of positive integers $%
(A_n)_{n\ge0}$ in three steps.

\noindent Step~1. Choose a sequence $(M_n)_{n\ge0}$ with all $M_n\in\mathbb{N%
}$.

\noindent Step~2. Let $(N_n)_{n\ge1}$ satisfy the recursion 
\begin{equation}  \label{EQ:N}
N_1=1,N_2=M_0, \ \text{ and }\ N_{n+2} =(M_nN_{n+1}+1)N_n \ \text{ for }\
n\ge1.
\end{equation}

\noindent Step~3. Define $(A_n)_{n\ge0}$ by 
\begin{equation}  \label{EQ:A}
A_0=M_0 \ \text{ and }\ A_n = M_nN_{n+1}+1\ \text{ for }\ n\ge1.
\end{equation}

Then there exists $(x_n)_{n\ge0}$ such that \eqref{EQ:equiv} is an
equivalence between an alternating series and a simple continued fraction,
namely, 
\begin{equation}  \label{EQ:AStoSCF}
\alpha := \sum_{n=0}^{\infty}\frac{(-1)^{n}}{A_0A_1\dotsb A_n} \cong
[0,M_0,M_1N_1,M_2N_2,M_3N_3,\dotsc].
\end{equation}

\noindent \textrm{(ii).} Conversely, if the continued fraction in %
\eqref{EQ:equiv} is simple for some $(x_n)_{n\ge0}$, then the sequence $%
(A_n)_{n\ge0}$ in \eqref{EQ:equiv} can be constructed by Steps 1, 2, 3.

\noindent \textrm{(iii).} The series in \eqref{EQ:AStoSCF} is the Pierce
expansion of $\alpha,$ that is, $A_{n+1}>A_n$ for $n\ge0$.

\noindent \textrm{(iv).} Distinct sequences $(M_n)_{n\ge0}\neq(M^{\prime
}_n)_{n\ge0}$ in Step~1 lead to distinct transcendental numbers $%
\alpha\neq\alpha^{\prime }$ in \eqref{EQ:AStoSCF}. In particular, if $%
\mathbb{S}$ denotes the set of real numbers $\alpha$ whose Pierce expansion
is equivalent to a simple continued fraction, then $\#\mathbb{S}%
=\aleph_0^{\aleph_0}=\mathfrak{c}$.
\end{corollary}

\begin{proof}
By definition, the continued fraction in \eqref{EQ:equiv} is simple if, and only if,
\begin{align*}
&{\rm(a).}\ x_0=1,\\
&{\rm(b).}\ A_nx_nx_{n+1}=1 \ \ \text{ for } \ n\ge0,\\
&{\rm(c).}\ A_0x_0\in\mathbb{N}, \ \text{ and}\\
&{\rm(d).}\ (A_n-1)x_n\in\mathbb{N} \ \text{ for } \ n\ge1.
\end{align*}

\noindent (i). Set $ x_0=1$ and $x_n=N_n/N_{n+1}$ for $n\ge1$. From formulas \eqref{EQ:A} and \eqref{EQ:N} we get $A_n=N_{n+2}/N_n$ for $n\ge1$. It is now easy to verify (a), (b), (c), and (d). Observing that $(A_n-1)x_n = M_nN_n$ for $n\ge1,$ the equivalence \eqref{EQ:equiv} gives \eqref{EQ:AStoSCF}. This proves (i).

\noindent (ii). Assume (a), (b), (c), and (d). Then $A_n\in\mathbb{N}$ implies that $x_n\in\mathbb{Q}$ for $n\ge1,$ so $x_n = N_n/D_n,$
where $N_n\in\mathbb{N}$ and $D_n\in\mathbb{N},$ with $\gcd(N_n,D_n)=1$. From (a) and (b), we get $N_1=1$ and $D_1=A_0$. From (d), we see that $D_n\mid (A_n-1)$ for $n\ge1,$ so there exists $M_n\in\mathbb{N}$ such that $A_n =M_nD_n+1$.
Since (b) implies \mbox{$A_nN_nN_{n+1}=D_nD_{n+1}$}, we get
\begin{equation} \label{EQ:DENNDD}
(M_nD_n+1)N_nN_{n+1}=D_nD_{n+1} \ \text{ for } \ n\ge1.
\end{equation}
Consequently, $N_{n+1}\mid D_nD_{n+1},$ so $N_{n+1}\mid D_n$. Also, $D_n\mid (M_nD_n+1)N_nN_{n+1},$ so $D_n\mid N_{n+1}$. Thus $D_n=N_{n+1}$ for all $n\ge1$; in particular, $N_2=D_1=A_0$. Making replacements in \eqref{EQ:DENNDD} and in $A_n =M_nD_n+1,$ we obtain \eqref{EQ:N} and \eqref{EQ:A}, respectively. This proves (ii).
 
\noindent (iii). Note that \eqref{EQ:N} and \eqref{EQ:A} give $A_{n+1} = M_{n+1}A_nN_n+1>A_n$ for $n\ge0$.

\noindent (iv). By Theorem~\ref{THM:altproductseries=SCF}, the sum $\alpha$ is transcendental. It now suffices to show that, given $\alpha=[0,M_0,M_1N_1,M_2N_2,\dotsc]$ and $\alpha'=[0,M'_0,M'_1N'_1,M'_2N'_2,\dotsc],$ if $\alpha=\alpha',$ then 
$M_n=M'_n$ for all~$n\ge0$. By the uniqueness of simple continued fraction expansion, $M_0=M'_0$ and $M_kN_k=M'_kN'_k$ for $k\ge1$. Using \eqref{EQ:N}, the rest of the proof is an easy induction, which we omit. This completes the proof of the corollary.
\end{proof}

\begin{example}
Choosing the constant sequence $M_n=1$ yields $N_1=1,N_2=1,$ and %
\mbox{$N_{n+2}=(N_{n+1}+1)N_n$} for $n\ge1$. Then $A_0=1$ and $A_n =
N_{n+1}+1$ for $n\ge1,$ so 
\begin{equation*}
A_n = 1,2, 3, 4, 9, 28, 225, 6076, 1361025, \dotsc 
\end{equation*}
(see \cite[A007704]{Sloane}). By (iv), we recover the transcendence of the 
\emph{Davison-Shallit constant} \cite[Example~A]{SD} (see also \cite[%
pp.~436, 445]{Finch}, \cite[A242724]{Sloane}) 
\begin{align*}
D: =&\ \sum_{n=0}^{\infty}\frac{(-1)^{n}}{A_0A_1\dotsb A_n} = 1-\frac12+%
\frac{1}{6}-\frac{1}{24}+\frac{1}{216}-\dotsb = 0.62946502045\dotso
\end{align*}
and, by (i), the expansion \cite[p.~122]{SD}, \cite[A006277]{Sloane} 
\begin{equation*}
D= [0,1,N_1,N_2,N_3,\dotso] = [0,1,1,1,2,3,8,27,224,6075,1361024,\dotsc]. 
\end{equation*}
\end{example}

\begin{example}
\label{EX:lambda} Let us define an \emph{alternating Liouville constant} by
the series 
\begin{align*}
\lambda := \sum_{n=2}^{\infty} \frac{(-1)^n}{10^{n!}} &= \frac{1}{10^2} - 
\frac{1}{10^6} + \frac{1}{10^{24}} - \frac{1}{10^{120}} + \dotsb \\
&=0.009999000000000000000000\underbrace{99\dotso9}_{96}00\dotso.
\end{align*}
For $n=1,2,3,\dotsc,$ the $n$th partial sum of the series satisfies 
\begin{align*}
\frac{P_n}{Q_n} := \sum_{k=2}^{n+1} \frac{(-1)^{k}}{10^{k!}} \implies 0 <
\left\vert \lambda - \frac{P_n}{Q_n} \right\vert < \frac{1}{10^{(n+2)!}} = 
\frac{1}{Q_n^{n+2}}.
\end{align*}
From this and \eqref{EQ:mu def}, we infer that $\lambda$ has irrationality
measure $\mu(\lambda) = \infty$. By definition, $\lambda$ is therefore a 
\emph{Liouville number}, so \emph{Liouville's theorem} \cite[\S 1.4]{BPSZ}, 
\cite[\S 9.3]{DuverneyBook}, \cite[\S 11.7]{HW} (or its descendant, the
Thue-Siegel-Roth theorem) implies $\lambda$ is transcendental.

On the other hand, its Pierce expansion 
\begin{align}  \label{EQ:lambda}
\lambda = \sum_{n=0}^{\infty}\frac{(-1)^{n}}{A_0A_1\dotsb A_n} = \frac{1}{%
10^{2!}} - \frac{1}{10^{2!}10^{3!-2!}} + \frac{1}{10^{2!}10^{3!-2!}10^{4!-3!}%
} - \dotsb
\end{align}
cannot be constructed from any sequence $(M_n)_{n\ge0}$ as in (i). (\emph{%
Proof}. If it could, then $M_0=A_0=10^{2!}$ would imply $%
M_110^{2!}+1=A_1=10^{3!-2!} $, contradicting $M_1\in\mathbb{N}$.)

Hence by (ii) a converse to Theorem~\ref{THM:altproductseries=SCF}, part
(iii), weaker than the false converse in Example~\ref{EX:e}, is also not
true. Namely, \emph{although $\lambda\not\in\mathbb{A}$ and $%
\mu(\lambda)\ge2.5$, the type~II series for $\lambda$ in \eqref{EQ:lambda}
is not equivalent to a simple continued fraction}.

More positively, one can show that, \emph{if a sequence $(M_n)_{n\ge0}$ in 
\textrm{(i)} grows sufficiently rapidly, then the sum $\alpha$ in the
equivalence \eqref{EQ:AStoSCF} is a Liouville number.}
\end{example}

The next section gives further applications of Theorem \ref%
{THM:altproductseries=SCF}.


\section{Sylvester's sequence and Cahen's constant.}

\label{SEC:Sylvester} 

There are not many ``naturally-occurring'' transcendental numbers for which
the simple continued fraction is known explicitly. They include the
beautiful expansions\newline

\begin{align*}
e-1&= [1,1,2,1,1,4,1,1,6,1,1,8,1,1,10,1,1,12,1,1,14,\dotsc], \\
\tan1&=[1,1,1,3,1,5,1,7,1,9,1,11,1,13,1,15,1,17,1,19,\dotsc], \\
1/\tanh1&=[1,3,5,7,9,11,13,15,17,19,21,23,25,27,29,31,33,\dotsc], \\
I_0(2)/I_1(2) &= [1,2,3,4,5,6,7,8,9,10,11,12,13,14,15,16,17,18,\dotsc],
\end{align*}
and those of $e^{2/q}, \tan\frac1q,\tanh\frac1q,$ and $I_{\frac{p}{q}%
}(\frac2q)/ I_{1+\frac{p}{q}}(\frac2q)$, for $p$ and $q$ in $\mathbb{N},$
where $I_c(x)$ is a modified, or hyperbolic, Bessel function of the first
kind \cite[Chapter~3]{DuverneyBook}. References to several others are given
in \cite[\S V]{SD}.

Theorem~\ref{THM:altproductseries=SCF} yields a doubly-infinite family of
such numbers. We define them by a natural recursion, independently of
Corollary~\ref{COR:(iii)}.

\begin{corollary}
\label{COR:SCFforCgen} Fix $k\in\mathbb{N}$ and $\ell\in\mathbb{N}$. For $%
n\ge0,$ define $s_n=s_n(k,\ell)$ by the recurrence 
\begin{equation}  \label{EQ:Sylvester_k}
s_0=k \ \text{ and } \ s_{n}=(s_0s_1\dotsb s_{n-1})^{\ell}+1 \ \text{ for }
\ n\ge1.
\end{equation}

\noindent\textrm{(i).} Then there is an equivalence 
\begin{align*}
C_{k,\ell}:=\sum_{n=0}^{\infty}\frac{(-1)^n}{s_{n+1}-1} \cong\
[a_0,a_1,a_2,\dotsc],
\end{align*}
where the partial quotients of the simple continued fraction are 
\begin{equation*}
\label{EQ:a_n} a_0=0,\ a_1=s_0^{\mspace{1mu}\ell},\ \text{ and }\
a_{n+1}=(s_n^{\ell}-1)\prod_{i=0}^{n-1}(s_i^{\ell})^{(-1)^{n+i}}\in\mathbb{N}%
\ \text{ for }\ n\ge1. 
\end{equation*}

\noindent\textrm{(ii).} The sum $C_{k,\ell}$ is transcendental, and $%
C_{k,\ell}=C_{k^{\prime },\ell^{\prime }}$ only when $(k,\ell)=(k^{\prime
},\ell^{\prime })$.

\noindent\textrm{(iii).} The double-exponential lower bound $a_n >
(k^{\ell}+1)^{(\ell+1)^{n-4}}$ holds for all $n\ge4$.

\noindent\textrm{(iv).} There are the summations 
\begin{align*}
\sum_{n=0}^{\infty} \frac{s_n^{\ell}-1}{s_{n+1}-1} = 1 \qquad \text{and}
\qquad \sum_{n=0}^{\infty} \frac{s_{2n+1}^{\ell}-1}{s_{2n+2}-1} &=
C_{k,\ell}.
\end{align*}

\noindent\textrm{(v).} Taking $\ell=1$ gives 
\begin{align*}
C_{k,1}=&\ \frac{1}{s_0} - \frac{1}{s_0s_1} + \frac{1}{s_0s_1s_2} - \frac{1}{%
s_0s_1s_2s_3} + \frac{1}{s_0s_1s_2s_3s_4} - \frac{1}{s_0s_1s_2s_3s_4s_5} +
\dotsb \\
\cong&\
[0,s_0,1,(s_0)^{2},(s_1)^{2},(s_0s_2)^{2},(s_1s_3)^{2},(s_0s_2s_4)^{2},
(s_1s_3s_5)^{2}, \dotsc].
\end{align*}

\noindent\textrm{(vi).} For odd $n\ge1$ and even $m\ge2,$ the partial
quotients $a_n$ and $a_m$ of $C_{k,1}$ are coprime.
\end{corollary}

\begin{proof}
\noindent(i).
Set $A_n:=s_n^{\ell}$ for $n\ge0$. Then \eqref{EQ:Sylvester_k} gives $s_{n+1}-1 = A_0A_1\dotsb A_n,$
so by Theorem~\ref{THM:altproductseries=SCF}, for any $x_0,x_1,\dotsc$ in $\mathbb{R^+}$ there is an equivalence
\begin{equation*}
C_{k,\ell} =\sum_{n=0}^{\infty}\frac{(-1)^{n}}{s_0^{\ell}s_1^{\ell}\dotsb s_n^{\ell}}
 \cong \frac{x_0}{s_0^{\ell}x_0+\displaystyle{\frac{s_0^{\ell}x_0x_1}{(s_1^{\ell}-1)x_1+\displaystyle{\frac{s_1^{\ell}x_1x_2}{(s_2^{\ell}-1)x_2+\displaystyle{ \cdot_{\displaystyle{\, \cdot}_{\displaystyle{\, \cdot}}}}}}}}}.
\end{equation*}
The partial numerators equal $1$ when $x_0=1$ and $x_{n+1}=(s_{n}^{\ell}x_{n})^{-1}$ for $n\ge0$. By induction, the solution of this recursion is
\begin{equation*} \label{EQ:x_n}
x_n=\prod_{i=0}^{n-1}(s_i^{\ell})^{(-1)^{n+i}}\ \text{ for } n\ge1.
\end{equation*}
The partial quotients are then $a_0=0,\ a_1=s_0^{\ell}x_0=s_0^{\ell},$ and $a_{n+1}=(s_n^{\ell}-1)x_n$ for $n\ge1$. Substituting $s_n^{\ell}-1=(s_0^{\ell}s_1^{\ell}\dotsb s_{n-1}^{\ell}+1)^{\ell}-1$ and expanding the binomial, the $1$s cancel, so $s_0^{\ell}s_1^{\ell}\dotsb s_{n-1}^{\ell}$ divides $s_n^{\ell}-1$ and $a_{n+1}\in\mathbb{N}$. This proves (i).

\noindent(ii).
Theorem~\ref{THM:altproductseries=SCF} and (i) imply $C_{k,\ell}\not\in\mathbb{A}$. From $a_1=k^{\ell}$ and $a_2=((k^{\ell}+1)^{\ell}-1)k^{-\ell},$ 
we deduce that $C_{k,\ell}\neq C_{k',\ell'}$ when $(k,\ell)\neq(k',\ell')$. This proves (ii).

\noindent(iii). Let $\alpha_n:=s_n-1$. Then \eqref{EQ:Sylvester_k} implies
$\alpha_{n+1} = \alpha_n(\alpha_n+1)^{\ell} > \alpha_n^{\ell+1}$ for $n\ge1$. As $\alpha_2 = k^{\ell}(k^{\ell}+1)^{\ell} \ge k^{\ell}+1,$ induction yields $\alpha_n \ge  (k^{\ell}+1)^{(\ell+1)^{n-2}}$ for $n\ge2$. Since (i) implies $a_n \ge s^{2\ell}_{n-3} > \alpha^{2\ell}_{n-3}\ge \alpha^{\ell+1}_{n-3},$ 
we get (iii).

\noindent(iv). For $n>0,$ definition \eqref{EQ:Sylvester_k} implies $s_{n+1}-1 = (s_n-1)s_n^{\ell},$ so
\begin{equation} \label{EQ:oddeven}
   \frac{1}{s_{n}-1} - \frac{1}{s_{n+1}-1} = \frac{s_n^{\ell}-1}{s_{n+1}-1}  \ \text{ for } \ n\ge1.
\end{equation}
Hence the first series in (iv) telescopes to $(s_0^\ell-1)(s_1-1)^{-1} + (s_1-1)^{-1}=1$.
Replacing $n$ with $2n+1$ in \eqref{EQ:oddeven}, we sum from $n=0$ to $\infty$ and obtain the second equality in (iv).

\noindent{\rm (v).} Set $\ell=1$ in parts (i) and (ii).

\noindent{\rm (vi).} Recursion \eqref{EQ:Sylvester_k} yields $\gcd(s_i,s_j)=1$ for $i\neq j,$ so (ii) follows from (i).
This completes the proof of the corollary.
\end{proof}

\begin{example}
\label{EX:SylCahen} Take $(k,\ell)=(1,1)$. \emph{Sylvester's sequence} \cite%
{Sylvester, Sylvester2} is defined as 
\begin{equation*}
(S_n)_{n\ge0} :=(s_{n+1}(1,1))_{n\ge0}= 2, 3, 7, 43, 1807, 3263443,
10650056950807,\dotsc \label{EQ: SS} 
\end{equation*}
(see \cite[p.~123]{SD}, \cite[pp.~436, 444]{Finch}, \cite{Golomb}, \cite[%
A000058]{Sloane}). Sylvester's sequence satisfies the recursion $S_0=2$ and $%
S_{n+1} = (S_n-1)S_n +1$ for $n\ge0$.

Likewise, $C:=C_{1,1}$ defines \emph{Cahen's constant} \cite{Cahen}, \cite[%
\S 6.7]{Finch}, \cite[A118227]{Sloane} 
\begin{align*}
C&=\sum_{n=0}^{\infty}\frac{(-1)^{n}}{S_n-1}=1-\frac12+\frac16-\frac{1}{42}+%
\frac{1}{1806}-\dotsb=0.643410546288338\dotso \\
&=1-\sum_{n=1}^{\infty}\frac{(-1)^{n-1}}{S_0S_1\dotsb S_{n-1} } = 1-\frac12+%
\frac{1}{2\cdot3}-\frac{1}{2\cdot3\cdot7}+\frac{1}{2\cdot3\cdot7\cdot43}%
-\dotsb.
\end{align*}
Corollary~\ref{COR:SCFforCgen} recovers $C\not\in\mathbb{A}$ from \cite{SD}
and gives the expansion \cite[A006279]{Sloane} 
\begin{align}  \label{EQ:SCFforC}
C&= [0,1,1,1,(S_0)^{2},(S_1)^2,(S_0S_2)^{2},(S_1S_3)^{2},
(S_0S_2S_4)^{2},(S_1S_3S_5)^{2}, \dotsc] \\
&=[0,\,1,\,1,\,1,\,2^2,\,3^2,\,14^2,\, 129^2,\, 25298^2,\,420984147^2,\,%
\dotso].  \notag
\end{align}
Since $\alpha_{n}:=S_{n}-1$ satisfies $\alpha_{n+1}-\alpha_n=\alpha_n^2$ and 
$\sum_{n=0}^{\infty}(-1)^n\alpha_n^{-1}=C,$ Proposition~\ref%
{THM:altproductseries=CF} and Theorem~\ref{THM:altproductseries=SCF} give,
respectively, the continued fractions 
\begin{equation*}
C = \frac{1}{1+\displaystyle{\frac{1^2}{1^2+\displaystyle{\frac{2^2}{2^2+%
\displaystyle{\frac{6^2}{6^2+\displaystyle{\frac{42^2}{42^2+\cdot_{%
\displaystyle{\, \cdot}_{\displaystyle{\, \cdot}}}}}}}}}}}} \cong\frac{1}{1+%
\displaystyle{\frac{1}{1+\displaystyle{\frac{2}{2+\displaystyle{\frac{3}{6+%
\displaystyle{\frac{7}{42+\cdot_{\displaystyle{\, \cdot}_{\displaystyle{\,
\cdot}}}}}}}}}}}}. 
\end{equation*}
\end{example}

In his $1891$ paper ``A remark on an expansion of numbers which has some
similarities with continued fractions'' \cite{Cahen}, Cahen defined $C$ and
showed that it is irrational. Exactly $100$ years later, as an example of
their ``self-similar'' (or ``self-generating'' \cite[\S 6.7]{Finch}, \cite[%
\S 6.7]{FinchErradd}) simple continued fractions, Davison and Shallit \cite%
{SD} proved that $C$ is transcendental and that $C=[0,1,q_0^2,q_1^2,q_2^2,%
\dotso]$. (This expansion agrees with \eqref{EQ:SCFforC}, by \eqref{EQ:qn}
and induction.) For generalizations of \cite{SD}, see Becker \cite{Becker}
and T\"{o}pfer \cite{Topfer}.

\begin{example}
\label{EX:Cahen} Corollary \ref{COR:SCFforCgen} shows that the Cahen-type
constant $C_{k,1} = [0, k, 1,\dotso],$ so 
\begin{equation*}
1 > C_{1,1} > \frac12 >C_{2,1} > \frac13 > C_{3,1} > \frac14 > C_{4,1} >
\dotsb. 
\end{equation*}
When $k=2$ we have $s_0(2,1)=2=s_1(1,1)=S_0$. It follows that in general $%
s_{n+1}(2,1)=s_{n+2}(1,1)=S_{n+1},$ so 
\begin{equation*}
C_{2,1}=\sum_{n=0}^{\infty}\frac{(-1)^{n}}{s_{n+1}(2,1)-1} =
\sum_{n=0}^{\infty}\frac{(-1)^{n}}{S_{n+1}-1} = 1-\sum_{n=0}^{\infty}\frac{%
(-1)^{n}}{S_n-1}=1-C. 
\end{equation*}
By Corollary \ref{COR:SCFforCgen}, 
\begin{align*}
C_{2,1}=[0,\,2,\,1,\,2^2,\,3^2,\,14^2,\,129^2,\,25298^2,\,420984147^2,\,%
\dotsc]\not\in\mathbb{A}.
\end{align*}
\end{example}

\begin{example}
For an example with $\ell >1,$ we take $(k,\ell)=(1,2)$ to get 
\begin{equation*}
(s_{n+1}(1,2))_{n\ge0}=2, 5, 101, 1020101, 1061522231810040101,\dotsc 
\end{equation*}
(see \cite{Shirali} and \cite[A231830]{Sloane}). Then $C_{1,2}$ is the
transcendental number 
\begin{align*}
C_{1,2}=\ &1 - \frac{1}{2^2} + \frac{1}{2^2\cdot5^2} - \frac{1}{%
2^2\cdot5^2\cdot101^2} + \dotsb \\
=\ & 1 - \frac{1}{4} + \frac{1}{100} - \frac{1}{1020100}+\dotsb
=0.759999019703\dotso.
\end{align*}
Here $\alpha_n:= s_{n+1}(1,2)-1$ satisfies $\alpha_{n+1}-\alpha_n =
\alpha_n^2(\alpha_n+2),$ so Proposition~\ref{THM:altproductseries=CF},
Theorem~\ref{THM:altproductseries=SCF}, and Corollary~\ref{COR:SCFforCgen}
give the continued fractions 
\begin{align*}
C_{1,2}&=\frac{1}{1+\displaystyle{\frac{1^2}{1^2\cdot3+\displaystyle{\frac{%
4^2}{4^2\cdot6+\displaystyle{\frac{100^2}{100^2\cdot102+\cdot_{\displaystyle{%
\, \cdot}_{\displaystyle{\, \cdot}}}}}}}}}} \cong \frac{1}{1+\displaystyle{%
\frac{1}{3+\displaystyle{\frac{4}{24+\displaystyle{\frac{25}{10200+\cdot_{%
\displaystyle{\, \cdot}_{\displaystyle{\, \cdot}}}}}}}}}} \\
&\cong
[0,1^2,2^2-1,(5^2-1)2^{-2},(101^2-1)2^25^{-2},(1020101^2-1)2^{-2}5^2101^{-2},
\\
&\quad\ \ \ (1061522231810040101^2-1)2^{2}5^{-2}101^{2}1020101^{-2},\dotsc]
\\
&= [0,\,1,\,3,\,6,\,1632,\,637563750,\, 1767398865801083661443214432,
\dotsc].
\end{align*}
\end{example}

Our final section studies series of \emph{positive} terms involving
Sylvester-type sequences.


\section{Some non-alternating series.}

\label{Kellogg and Curtiss}

Another series formed from Sylvester's sequence is the sum of reciprocals.
Setting $(k,\ell)=(1,1)$ in \eqref{EQ:oddeven}, the right-hand side is then $%
S_n^{-1},$ so the series telescopes to 
\begin{equation}  \label{EQ:telescope1}
\sum_{n=0}^{\infty}\frac{1}{S_n}= \sum_{n=0}^{\infty}\left(\frac{1}{S_n-1} - 
\frac{1}{S_{n+1}-1}\right)= \frac{1}{S_0-1} = 1,
\end{equation}
a rational number. By contrast, the corresponding alternating sum 
\begin{equation*}
\sum_{n=0}^{\infty}\frac{(-1)^{n}}{S_n} = \sum_{n=0}^{\infty}\frac{(-1)^{n}}{%
S_n-1} - \sum_{n=0}^{\infty}\frac{(-1)^{n}}{S_{n+1}-1} = C-(1-C) =2C-1 
\end{equation*}
is transcendental, as are the non-alternating sums 
\begin{align}  \label{EQ:greedyC}
\sum_{n=0}^{\infty}\frac{1}{S_{2n}}&=\sum_{n=0}^{\infty}\left(\frac{1}{%
S_{2n}-1}-\frac{1}{S_{2n+1}-1}\right)=C
\end{align}
and $\sum_{n=0}^{\infty}S_{2n+1}^{-1}= 1-C$.

Finch asked, ``What can be said about $\sum_{n=0}^{%
\infty}(S_n-1)^{-1}=1.6910302067\dotso$?'' \cite[p.~436]{Finch}. We denote
this constant by 
\begin{equation*}
K:=\sum_{n=0}^{\infty}\frac{1}{S_n-1} = 1 + \sum_{n=1}^{\infty}\frac{1}{%
S_0S_1\dotsb S_{n-1}} 
\end{equation*}
and we name it the \emph{Kellogg-Curtiss constant}, because Kellogg
conjectured \cite{Kellogg}, and Curtiss proved \cite{Curtiss}, the following
bound on solutions to a unit fraction equation: 
\begin{equation*}
x_i\in\mathbb{N}\ \text{ and }\ \sum_{i=0}^n\frac{1}{x_i} = 1 \quad \implies
\quad \max_{0\le i\le n} x_i\le S_n-1. 
\end{equation*}

\begin{remark}
By \eqref{EQ:telescope1}, one solution of the equation $\sum_{i=0}^{%
\infty}x_i^{-1} = 1$ is $x_i=S_i$. In fact, this is the solution provided by
the ``greedy Egyptian fraction algorithm''---see Soundararajan \cite{Sounda}%
. Likewise, the greedy Egyptian fraction expansion of Cahen's constant~$C$
is series \eqref{EQ:greedyC} with $x_i=S_{2i}$.
\end{remark}

The following general result shows in particular that $K$ is irrational.

\begin{proposition}
\label{PROP:K_k} For $k\in\mathbb{N}$ and $\ell\in\mathbb{N},$ define the 
\emph{Kellogg-Curtiss-type constant} 
\begin{equation*}
\label{EQ:K_k} K_{k,\ell}:=\sum_{n=0}^{\infty}\frac{1}{s_{n+1}(k,\ell)-1}, 
\end{equation*}
where the Sylvester-type sequence $(s_n(k,\ell))_{n\ge0}$ is defined in
Corollary~\ref{COR:SCFforCgen}.

\noindent\textrm{(i).} Then $K_{k,\ell}\not\in\mathbb{Q}$. In particular,
the Kellogg-Curtiss constant $K=K_{1,1}=1+K_{2,1}$ is irrational.

\noindent\textrm{(ii).} If $\ell\ge2,$ then $K_{k,\ell}$ is transcendental
and $\mu(K_{k,\ell})\ge3$.
\end{proposition}

We could prove (i) from the fact that, \emph{given a non-decreasing sequence
of positive integers $A_0,A_1,\dotsc,$ the \emph{Engel series} $%
\sum_{n=0}^{\infty}(A_0A_1\dotsb A_n)^{-1}$ converges to an irrational
number if} (\emph{and only if}) $A_n$ \emph{tends to infinity with} $n$
(see, e.g., \cite[\S 2.2]{DuverneyBook}). Instead, we give a mostly
self-contained proof. It uses partial sums instead of continued fractions
(compare to Example~\ref{EX:lambda}).

\begin{proof}[Proof of Proposition \ref{PROP:K_k}]
Let us fix integers $k\ge1$ and $\ell\ge1,$ and write $s_{n}$ in place of $s_{n}(k,\ell)$. Then for $n\ge1,$ the $n$th partial sum of the series for $K_{k,\ell}$ is, in lowest terms,
 $$\frac{P_n}{Q_n}:= \sum_{i=0}^{n-1}\frac{1}{s_{i+1}-1}=\sum_{i=0}^{n-1}\frac{1}{s_0^{\ell}s_1^{\ell}\dotsb s_{i}^{\ell}}\quad  \implies \quad Q_n=s_0^{\ell}s_1^{\ell}\dotsb s_{n-1}^{\ell} = s_{n}-1.$$ 
With this value of $Q_n$ we see that
\begin{align} \label{EQ:psum}
 \begin{split}
0<K_{k,\ell}-\frac{P_n}{Q_n}&= \sum_{i=n}^{\infty}\frac{1}{s_0^{\ell}s_1^{\ell}\dotsb s_{i}^{\ell}}=
\frac{1}{Q_n}\sum_{j=0}^{\infty}\frac{1}{s_{n}^{\ell}\dotsb s_{n+j}^{\ell}}\\
&<\frac{1}{Q_n}\sum_{j=0}^{\infty}\frac{1}{(s_{n}^{\ell})^{j+1}}
=\frac{1}{Q_n}\frac{1}{s_{n}^{\ell}-1} 
\le \frac{1}{Q_n^{\ell+1}}.
 \end{split}
\end{align}

\noindent{\rm (i).} If $K_{k,\ell}\in\mathbb{Q},$ say $K_{k,\ell}=P/Q,$ then
\begin{equation} \label{EQ:contrad}
K_{k,\ell} - \frac{P_n}{Q_n} = \frac{P}{Q} - \frac{P_n}{Q_n}  \ge \frac{1}{QQ_n} > \frac{1}{Q^2_n}
\end{equation}
for $n$ so large that $Q_n>Q$. But $\ell \ge1$, so \eqref{EQ:contrad} contradicts \eqref{EQ:psum}. Therefore, $K_{k,\ell}\not\in\mathbb{Q}$.

\noindent{\rm (ii).} From \eqref{EQ:psum}, we infer that $\mu(K_{k,\ell})\ge\ell+1$. If $\ell\ge2,$ then $\mu(K_{k,\ell}) \ge 3$, so by the Thue-Siegel-Roth theorem, $K_{k,\ell}\not\in\mathbb{A}$. This completes the proof of the proposition.
\end{proof}

By a similar argument (also not using Theorem~\ref{THM:altproductseries=SCF}
or continued fractions), $C_{k,\ell}\not\in\mathbb{A}$ for $\ell\ge2$. The
case $\ell=1$ though (which includes Cahen's constant $C$) would seem to
require using Theorem~\ref{THM:altproductseries=SCF}, as in the proof of
Corollary~\ref{COR:SCFforCgen}. However, Duverney \cite{DuverneyUnpub} has
found a proof that $C\not\in\mathbb{A}$ which is similar to that of
Proposition~\ref{PROP:K_k}, part (ii). He uses relation \eqref{EQ:greedyC}
and the fact that $S_{2n+2}>\frac18S_{2n}^4,$ which follows from $%
S_{n+1}>\frac12S_{n}^2$.

Duverney has also answered Finch's question by pointing out that, as a
special case of a result of Becker \cite[p.~186, Remark~(ii)]{Becker}, \emph{%
the Kellogg-Curtiss constant $K$ is transcendental}.

\begin{conjecture}
For $k\ge1$, the Kellogg-Curtiss-type constant $K_{k,1}$ is transcendental.
\end{conjecture}

\begin{acknowledgment}{Acknowledgments.}
I thank Daniel Duverney and Michael Nyblom for (independently) pointing out the series which shows that Sierpi\'{n}ski's theorem is sharp. I am also indebted to Duverney for generalizing my earlier special case of Corollary~\ref{COR:(iii)}. I thank Steven Finch for comments on the manuscript and for discussions on \cite{SD}. Finally, I am grateful to Yohei Tachiya for a proof that $K$ is irrational.
\end{acknowledgment}

\begin{affil}
This manuscript was submitted posthumously.\\
The author passed away on January 16, 2020.
\end{affil}

\end{document}